\documentclass[runningheads]{CMSIM}

\usepackage{graphicx} 

\usepackage{amsmath,amssymb}
\newcommand{\Ii}{ \mathbf{1} }

\renewcommand{\thefootnote}

\title*{Extremes in Random Graphs Models of Complex Networks}

\titlerunning{\it Extremes in Random Graphs}

\author{
Natalia Markovich\inst{1}
}

\authorrunning{\it Markovich}

\institute{
V.A.Trapeznikov Institute of Control Sciences,
                Russian Academy of Sciences, Profsoyuznaya 65, 117997 Moscow, Russia,\\
Moscow Institute of Physics and Technology State University,
Kerchenskaya 1À, 117303 Moscow, Russia,
 \\(E-mail: nat.markovich@gmail.com)
}
\begin{document}
\thispagestyle{empty}
\maketitle
\setlength{\leftskip}{0pt}
\setlength{\headsep}{16pt}
\begin{abstract}
Regarding the analysis of Web communication, social and complex networks the fast finding of most influential nodes in a network graph constitutes an important research problem.
We use two indices of the influence of those nodes, namely, PageRank and a Max-linear model. We consider the PageRank 
as
an autoregressive process with a random number of random coefficients that depend on ranks of incoming nodes and their out-degrees and assume that the coefficients are independent and distributed with regularly varying tail and with the same tail index. Then  it is proved that the tail index and the extremal index are the same for both PageRank and the Max-linear model and the values of these indices are found. The achievements are based on the study of random sequences of a random length and the comparison of the distribution of their maxima and linear combinations.
\end{abstract}

\section{Introduction}
Regarding the analysis of Web communication, social and complex networks the fast finding of most influential nodes in a network graph constitutes an important research problem. PageRank remains the most popular characteristic of such influence. We aim to find an extremal index of  PageRank whose reciprocal value determines the first hitting time, i.e. a minimal time to reach the first influential node by means of a PageRank random walk. The extremal index $\theta\in[0,1]$ has many other interpretations and  plays a significant role in the theory of extreme values. Particularly, the limit distribution of maxima of stationary random variables (r.v.s) depends on $\theta$. For independent r.v.s $\theta=1$ holds.
\\
$\theta$ has a connection to the tail index that shows the heaviness of the tail of a stationary distribution of an underlying process.
\\
Google's PageRank defines the rank $R(X_i)$ of the Web page $X_i$  as
\begin{equation}\label{0}R(X_i)=c\sum_{X_j\in N(X_i)}\frac{R(X_j)}{D_j}+(1-c)q_i,\qquad i=1,...,n,\end{equation}
where $N(X_i)$ is the set of pages that link to $X_i$ (in-degree), $D_j$ is the number of outgoing links of page $X_j$ (out-degree), $c\in(0,1)$  is a damping factor, $q=(q_1,q_2,...,q_n)$ is a personalization probability vector or user preference such that $q_i\ge 0$ and $\sum_{i=1}^nq_i=1$, and  $n$ is the total number of pages, \cite{LangMey:06}. We omit in (\ref{0}) the term with dangling nodes for simplicity.
\\
 PageRank of a randomly selected page (a node in the graph) with random in- and out-degrees may be considered as
a branching process (Cf. \cite{ChLiOl:14}, \cite{Jel:10},  \cite{Vol:10})
 \begin{equation}\label{6}R_{i}=\sum_{j=1}^{N_i}A_{j}R_i^{(j)}+Q_i,~~ i=1,...,n,\end{equation}
denoting $R_i=R(X_i)$, $A_{j}=^dc/D_j$, $Q_i=(1-c)q_i$, \cite{Vol:10}. $R_i^{(j)}$ are ranks of descendants 
of node $i$, i.e. nodes with incoming links to node $i$.
The 
r.v. $N_i$ determines an in-degree, i.e. a number of directed edges to the $i$th node, and a number of nodes in the first generation of descendants belonging to the $i$th node as a parent,   $\{Q_i\}$  is a sequence of i.i.d. r.v.s.
\\
Starting from the initial page (node) $X_0$, a PageRank random walk determines a regenerative process or Harris recurrent process $\{X_t\}$, letting it visits pages-followers of the underlying node with probability $c$ and it restarts with probability $1-c$ by jumping to a random independent node.
\\
A Max-linear model can be considered as an alternative characteristic of the node influence. This model is obtained by a substitution of sums in Google's definition of PageRank by maxima, i.e.
\begin{equation}\label{7}R_{i}=\bigvee_{j=1}^{N_i}A_{j}R_i^{(j)}\vee Q_i,~~ i=1,...,n,\end{equation}
is proposed in \cite{GisKlu:15}.
\\
Formally, (\ref{6}) can be considered as an autoregressive process  with the random number $N_i$ of random coefficients and the independent random term $Q_i$. The extremal index of $AR(1)$ processes with regularly varying stationary distribution and its relation to the tail index were considered in
\cite{Haan:1989}. The extremal index of  $AR(q)$, $q\ge 1$ processes with $q$ random coefficients was obtained in \cite{KluPer:07} in a form which is not convenient for calculations. In \cite{Gold:13} the results by \cite{Haan:1989} were extended to multivariate regularly varying distributed random sequences and the extremal  and tail indices of sum and maxima of such sequences with $l\ge 1$ r.v.s were derived.
\\
Our achievements extend and adapt the results by \cite{Gold:13} to PageRank and  Max-linear processes. The problem concerns the finding of the extremal index of a  random graph that models a real network where incoming nodes of the root node  may be linked and, hence, be dependent. Such a random graph is called a Thorny Branching Tree (TBT) since any node may have outbound stubs (teleportations) to arbitrary nodes of the network, \cite{ChLiOl:14}.  In this respect, such a graph cannot be considered as a pure Galton-Watson branching process where descendants of any node are mutually independent and teleportations are impossible.
\\
The paper is organized as follows. In Section \ref{Sec1} we recall necessary 
results regarding the relation between the tail  and extremal indices obtained in \cite{Gold:13} for multivariate random sequences which are regularly varying distributed (Theorems \ref{T1} and \ref{T2}).
Linear combinations and maxima of the random sequences of a fixed length are considered and it is derived that they have the same tail and extremal indices.
In Section \ref{Sec2} we extend Theorem \ref{T2} to the case of unequal tail indices assuming r.v.s of a random sequence (Theorem \ref{T3}).
In Section \ref{Sec3} we consider sequences of random lengths and obtain the tail and extremal indices of their linear combinations and maxima (Theorem \ref{T4}). We further discuss how these results can be applied to  PageRank and the Max-linear processes
in Section \ref{Sec4}.
\section{Related Work}\label{Sec1}
Let $\{R_j\}$ be a stationary sequence with distribution function $F(x)$ and maxima $M_n=\max_{1\le j\le n}R_j$. We shall  interpret $\{R_j\}$ as PageRanks of Web pages.
\begin{definition}\label{Def1}A stationary sequence  $\{R_n\}_{n\ge 1}$ is said to have extremal index
$\theta\in[0,1]$ if
for each $0<\tau <\infty$ there is a sequence of real numbers $u_n=u_n(\tau)$ such that
\begin{equation}\label{1}\lim_{n\to\infty}n(1-F(u_n))=\tau \qquad\mbox{and}\end{equation}
\begin{equation}\label{2}\lim_{n\to\infty}P\{M_n\le u_n\}=e^{-\tau\theta}\end{equation}
hold (\cite{Lead:83}, p.53).
\end{definition}
In \cite{Gold:13} the following theorems are proved which we will use to find the extremal and tail indices of PageRank and a Max-linear model.
Let $Y_n^{(1)},Y_n^{(2)},...,Y_n^{(l)}$, $n\ge 1$, $l\ge 1$ be sequences of r.v.s having stationary distributions with tail indices $k_1,...,k_l$ and extremal indices $\theta_1,...,\theta_l$, respectively, i.e.
\begin{eqnarray*}
P\{Y_n^{(i)}>x\}&\sim &c^{(i)}x^{-k_i}\qquad \mbox{as}\qquad x\to\infty,\end{eqnarray*}
where $c^{(i)}$ are some real positive constants.
\\
Let us consider the weighted sum
\begin{eqnarray}\label{8}
Y_n(z)&=&z_1Y_n^{(1)}+z_2Y_n^{(2)}+...+z_lY_n^{(l)}, \qquad z_1,...,z_l>0
\end{eqnarray}
 and denote its tail index by $k(z)$ and extremal index by $\theta(z)$. Supposing that there is a minimal tail index among $k_1,...,k_l$, the following theorem states the corresponding $k(z)$ and  $\theta(z)$.
\begin{theorem}\label{T1}  (\cite{Gold:13}) Let $k_1<k_i$, $i=2,...,l$ hold. Then $Y_n(z)$ has the tail index $k(z)=k_1$ and the extremal index $\theta(z)=\theta_1$.
\end{theorem}
In the next theorem it is assumed that sequences $Y_n^{(1)},Y_n^{(2)},...,Y_n^{(l)}$ are mutually independent with equal tail indices $k_1=...=k_l=k$. We denote
\begin{eqnarray}\label{20}Y_n^*(z)&=&\max\left(z_1Y_n^{(1)},z_2Y_n^{(2)},...,z_lY_n^{(l)}\right).\end{eqnarray}
\begin{theorem} \label{T2} (\cite{Gold:13}) The sequences $Y_n^*(z)$ and $Y_n(z)$ have the same tail index $k$ and the same extremal index equal to
\begin{eqnarray*}\theta(z)&=&
\frac{c^{(1)}z_1^k}{{c^{(1)}z_1^k}+...+c^{(l)}z_l^k}\theta_1+...+\frac{c^{(l)}z_l^k}{{c^{(1)}z_1^k}+...+c^{(l)}z_l^k}\theta_l.
\end{eqnarray*}
\end{theorem}
 \section{Generalization of Theorem \ref{T2}}\label{Sec2}
 Theorem \ref{T3} is a generalization of Theorem \ref{T2} to the case of unequal tail indices.
\begin{theorem}\label{T3} Let $\{Y_n^{(j)}\}$, $n\ge 1$, $j=1,..., l$  be mutually independent regularly varying r.v.s 
with tail indices $k_1,...,k_{l}$, respectively.  Let $k_m<k_i$, $i=1,...,l$, $i\neq m$ hold. Then r.v.s $Y_n^*(z)$ and $Y_n(z)$ have the same tail index $k(z)=k_m$ and the same extremal index $\theta(z)=\theta_m$.
\end{theorem}
\begin{proof}
First we show that
\begin{eqnarray}\label{13}P\{Y_n^*(z)>x\}&\sim & c(z) x^{-k_m},\qquad x\to\infty,
\end{eqnarray}
where $c(z)=\sum_{i=1}^lc^{(i)}z_i^{k_i}\Ii\{k_i=k_m\}$. Similar to \cite{Gold:13} and as
\begin{eqnarray}\label{11}P\{z_iY_n^{(i)}>x\}&\sim & 
c^{(i)}z_i^{k_i} x^{-k_i}\end{eqnarray}
holds, we have
\begin{eqnarray}\label{4}&&P\{Y_n^*(z)>x\}= P\{\max(z_1Y_n^{(1)},..,z_lY_n^{(l)})>x\}\nonumber
\\
&=& 1-P\{\max(z_1Y_n^{(1)}\le x\}\cdot...\cdot P\{z_lY_n^{(l)})\le x\}\nonumber
\\
&=& \sum_{i=1}^lP\{z_iY_n^{(i)}>x\}\nonumber
\\
&+&\sum_{k=2}^l(-1)^{k-1}\sum_{i_1<i_2<...<i_k;i_1,i_2,...,i_k=1}^lP\{z_{i_1}Y_n^{(i_1)}>x\}\cdot...\cdot P\{z_{i_k}Y_n^{(i_k)}>x\}\nonumber
\\
&\sim & \sum_{i=1}^lc^{(i)}z_i^{k_i}x^{-k_i}\nonumber
\\
&+&\sum_{k=2}^l(-1)^{k-1}\sum_{i_1<i_2<...<i_k;i_1,i_2,...,i_k=1}^lc^{(i_1)}z_{i_1}^{k_{i_1}}x^{-k_{i_1}}\cdot...\cdot c^{(i_k)}z_{i_k}^{k_{i_k}}x^{-k_{i_k}}\nonumber
\\
&\sim &c(z)x^{-k_m}+o(x^{-k_m}),\qquad x\to\infty.
\end{eqnarray}
Thus, $P\{Y_n(z)>x\}\sim  c(z) x^{-k_m}$ follows from Theorem \ref{T1}.
\\
Now we show that $Y_n^*(z)$ and $Y_n(z)$ have the same extremal index $\theta(z)=\theta_m$. We use the same notations as in \cite{Gold:13}
\begin{eqnarray*}M_n^{(i)}&=& \max\{Y_1^{(i)}, Y_2^{(i)},...,Y_n^{(i)}\}, ~i=1,..,l;
\\
M_n(z)&=& \max\{Y_1(z), Y_2(z),...,Y_n(z)\},
\\
M_n^*(z)&=& \max\{Y_1^*(z), Y_2^*(z),...,Y_n^*(z)\}, ~n\ge 1.
\end{eqnarray*}
By (\ref{20}) it holds
\begin{eqnarray*}M_n^*(z)&=& \max\{z_1Y_1^{(1)},..., z_1Y_n^{(1)},..., z_lY_1^{(l)},..., z_lY_n^{(l)}\}
\\
&=& \max\{z_1M_n^{(1)},...,z_lM_n^{(l)}\}.
\end{eqnarray*}
Then we get
\begin{eqnarray}\label{12}
P\{M_n^*(z)n^{-1/k}\le x\}&=& P\{z_1M_n^{(1)}n^{-1/k}\le x,...,z_lM_n^{(l)}n^{-1/k}\le x\}
\end{eqnarray}
Since $k_m$ is the minimal tail index we have
\begin{eqnarray*}
P\{z_iM_n^{(i)}n^{-1/k_m}\le x\}&=& P\{z_iM_n^{(i)}n^{-1/k_i}\le xn^{1/k_m-1/k_i}\}.
\end{eqnarray*}
It implies
\begin{eqnarray}\label{22}z_iM_n^{(i)}n^{-1/k_m}\to^P 0, ~~i=1,...,l, ~~i\neq m ~~\mbox{as} ~~n\to\infty\end{eqnarray} since $\lim_{n\to\infty}P\{z_iM_n^{(i)}n^{-1/k_i}\le x\}= \exp(-c^{(i)}\theta_iz_i^{k_i}x^{-k_i})$. By (\ref{12}) it holds
\begin{eqnarray*}
P\{M_n^*(z)n^{-1/k_m}\le x\}&\to & \exp(-c^{(m)}z_m^{k_m}\theta_mx^{-k_m}),\qquad n\to\infty.
\end{eqnarray*}
Now we have to show that $P\{M_n^*(z)n^{-1/k_m}\le x\}\sim P\{M_n(z)n^{-1/k_m}\le x\}$. Let us denote $u_n=xn^{1/k_m}$. Note that the event $\{M_n^*(z)\le u_n\}$ follows from  $\{M_n(z)\le u_n\}$. Then, as in \cite{Gold:13}, we obtain
\begin{eqnarray}\label{26}&&0\le P\{M_n^*(z)\le u_n\}-P\{M_n(z)\le u_n\}
\\
&=&P\{M_n^*(z)\le u_n\}-P\{M_n^*(z)\le u_n, M_n(z)\le u_n\}\nonumber
\\
&=& P\{M_n^*(z)\le u_n, M_n(z)> u_n\}\le\sum_{k=1}^nP\{M_n^*(z)\le u_n, Y_k(z)> u_n\}\nonumber
\\
&\le &\sum_{k=1}^nP\{Y_k^*(z)\le u_n, Y_k(z)> u_n\}=nP\{Y_n^*(z)\le u_n, Y_n(z)> u_n\}\nonumber
\end{eqnarray}
due to the stationarity of the sequences $Y_n^*(z)$ and $Y_n(z)$.
Lemma 1 in \cite{Gold:13} states that
\begin{eqnarray}\label{27}
P\{Y_n^*(z)\le u_n|Y_n(z)> u_n\}\to 0, \qquad n\to\infty,
\end{eqnarray}
for i.i.d. regularly varying $\{Y_n^{(j)}\}$ with equal tail index. This can be extended to the case of unequal $k_1,...,k_l$.
Since
\begin{eqnarray}\label{29}
nP\{Y_n(z)> u_n\}&\to &c(z)x^{-k_m}, \qquad n\to\infty,
\end{eqnarray}
and (\ref{27}) hold,
it follows
\begin{eqnarray*}&&\lim_{n\to\infty}\left(P\{M_n^*(z)\le u_n\}-P\{M_n(z)\le u_n\}\right)=0.
\end{eqnarray*}
\end{proof}
\section{Extremal Index of PageRank and the Max-Linear Processes}\label{Sec3}
We denote in (\ref{6}) $R_i$ as $Y_i(z)$ and $A_jR_i^{(j)}=cR_i^{(j)}/D_j$,
$j=1,..., N_i$ as $z_jY_i^{(j)}$. 
Then we can represent (\ref{6}) in the form (\ref{8}) as
\begin{eqnarray}\label{9}Y_i(z)&=&\sum_{j=1}^{N_i}z_jY_i^{(j)}+Q_i,\qquad i=1,...,n,
\end{eqnarray}
where $N_i$ is a nonnegative integer-valued r.v.. 
In the context of PageRank $z_j=c$, $j=1,2,...,N_i$, $Q_i=z^*q_i$ with $z^*=1-c$ and $N_i$ represents the node in-degree. It is realistic to assume that $N_i$ is a power law distributed r.v. 
with parameter 
$\alpha>0$, i.e. \begin{eqnarray}\label{10}P\{N_i=\ell\}\sim \ell^{-\alpha}
\end{eqnarray}
and $N_i$ is bounded by a total number of nodes in the network. 
\\
The distribution of $N_i$ is in the domain of attraction of the Fr\'{e}chet distribution with shape parameter $\alpha>0$ and $P\{N_i>x\}=x^{-\alpha}\ell(x)$, $\forall x>0$, where $\ell(x)$ is a slowly varying function, since it satisfies a sufficient condition for this property, i.e. the von Mises type condition $\lim_{n\to\infty}nP\{N_i=n\}/P\{N_i>n\}=\alpha$, \cite{Ander80}.
\\
 Theorem \ref{T4} is an extension of Theorems \ref{T2} and \ref{T3} to maxima and sums of multivariate random sequences of random lengths, that can be applied to PageRank and the Max-linear processes.
Let us turn to (\ref{9}) and denote
\begin{eqnarray*}Y_{N_n}^*(z)&=&\max(z_1Y_n^{(1)},..,z_{N_n}Y_n^{(N_n)}, Q_n),
\\ Y_{N_n}(z)&=&z_1Y_n^{(1)}+..+z_{N_n}Y_n^{(N_n)}+Q_n.\end{eqnarray*}

\begin{theorem}\label{T4}Let $\{Y_n^{(j)}\}$, $n\ge 1$, $j=1,..., N_n$ and $q_n=Q_n/z^*$ be mutually independent regularly varying i.i.d. r.v.s 
with tail 
indices $k>0$ 
and $\beta>0$, 
respectively,
and 
$N_n$ be regularly varying r.v. with tail index $\alpha>0$. Let $Y_n^{(1)}$,.., $Y_n^{(N_n)}$ 
have extremal indices $\theta_1,...,\theta_{N_n}$, respectively. 
Then r.v.s $Y_{N_n}^*(z)$ and $Y_{N_n}(z)$ are regularly varying distributed with the same tail index $k(z)=\min(k,\alpha, \beta)$  
and  the same extremal index $\theta(z)$ such that
\begin{eqnarray}\label{33}
\theta(z) &=& (z^*)^{\beta}, 
~~ \mbox{if}~~k\ge\beta,\nonumber
\\
\theta(z)&=& \sum_{i=1}^{\infty}c^{(i)}\theta_i z_i^{k}/c(z), 
 ~~ \mbox{if}~~k< \beta,
\end{eqnarray}
where
$c(z)=\sum_{i=1}^{\infty}c^{(i)}z_i^{k}$ holds.
\end{theorem}
\begin{proof} 
We shall show first that
\begin{eqnarray}\label{3}P\{Y_{N_n}^*(z)>x\}\sim P\{Y_{N_n}(z)>x\}\sim 
x^{-\min(k,\alpha,\beta)}.\end{eqnarray}
Since r.v.s $\{Y_n^{(j)}\}_{j\ge 1}$ are subexponential and i.i.d. it holds
\begin{eqnarray}\label{4a}P\{z_1Y_n^{(1)}+..+z_{\lfloor x\rfloor}Y_n^{(\lfloor x\rfloor)}>x\}&\sim & P\{\max(z_1Y_n^{(1)},..,z_{\lfloor x\rfloor}Y_n^{(\lfloor x\rfloor)})>x\}\nonumber
\\
&\sim &xP\{z_1Y_n^{(1)}>x\},\qquad x\to\infty,
\end{eqnarray}
\cite{GolKlu:98}.
Due to mutual independence of $Q_n$ and $\{Y_n^{(j)}\}$ and similar to (\ref{4})
we get
\begin{eqnarray}\label{16a}&&P\{Y_{N_n}^*(z)>x\}= P\{Y_{N_n}^*(z)>x, N_n\le x\}+P\{Y_{N_n}^*(z)>x, N_n>x\}\nonumber
\\
&\le & P\{Y_{\lfloor x\rfloor}^*(z)>x\}+P\{N_n>x\}\nonumber
\\
&=& 1-P\{\max(z_1Y_n^{(1)},..,z_{\lfloor x\rfloor}Y_n^{(\lfloor x\rfloor)})\le x\}P\{Q_n\le x\}+P\{N_n>x\}\nonumber
\\
&\sim & c_Nx^{-\alpha}+c_q(z^*)^{\beta}x^{-\beta}+c(z)x^{-k}\sim x^{-\min\{k,\alpha,\beta\}},
\end{eqnarray}
as $x\to\infty$, where $c_N, c_q>0$, $c(z)=\sum_{i=1}^{\infty}c^{(i)}z_i^{k}$.
On the other hand,
\begin{eqnarray}\label{15a}&&P\{Y_{N_n}^*(z)>x\}\ge
0+P\{Y_{N_n}^*(z)>x, N_n>x\}\nonumber
\\
&\ge & P\{Y_{\lceil x\rceil}^*(z)>x\}+P\{N_n>x\}+ P\{Y_{\lceil x\rceil}^*(z)\le x, N_n\le x\}-1\nonumber
\\
&\sim & x^{-\min\{k,\alpha,\beta\}}
\end{eqnarray}
holds, since $P\{Y_{\lceil x\rceil}^*(z)\le x, N_n\le x\}\to 1$ as $x\to\infty$. Due to (\ref{16a}) and (\ref{15a}) we obtain \begin{eqnarray*}P\{Y_{N_n}^*(z)>x\}&\sim &x^{-\min\{k,\alpha,\beta\}}.\end{eqnarray*}
The same is valid for $Y_{N_n}(z)$ by substitution of the maximum by the sum due to (\ref{4a}). 
Hence, (\ref{3}) follows.
\\
Let us prove that $Y_{N_n}^*(z)$ and $Y_{N_n}(z)$ have the same extremal index $\theta(z)$. 
Let us denote
\begin{eqnarray}\label{17a}
&&M_{N_n}^*(z)= \max\{Y_{N_1}^*(z), Y_{N_2}^*(z),...,Y_{N_n}^*(z)\}
\\
&=&\max\{z_1Y_1^{(1)},...,z_{N_1}Y_1^{(N_1)}, Q_1, ...,z_1Y_n^{(1)},...,z_{N_n}Y_n^{(N_n)}, Q_n\}\nonumber
\end{eqnarray}
and
\begin{eqnarray*}
M_{N_n}(z)&=& \max\{Y_{N_1}(z), Y_{N_2}(z),...,Y_{N_n}(z)\}
\\
&=&\max\{z_1Y_1^{(1)}+...+z_{N_1}Y_1^{(N_1)}+Q_1, ...,z_1Y_n^{(1)}+...+z_{N_n}Y_n^{(N_n)}+Q_n\}.
\end{eqnarray*}
Without loss of generality we may assume that $N_{n}=\max\{N_1,...,N_n\}$. Then we can complete vectors $(z_1Y_i^{(1)},...,z_{N_i}Y_i^{(N_i)})$, $i=1,2,...,n$ by zeros up to the dimension $N_n$ and separate the vector $(Q_1,...,Q_n)$. We rewrite (\ref{17a}) as
\begin{eqnarray*}
&&M_{N_n}^*(z)=\max\{z_1Y_1^{(1)},...,z_1Y_n^{(1)},..., z_{N_n}\cdot 0,...,z_{N_n}\cdot 0,..., z_{N_n}Y_n^{(N_n)},
\\
&&
Q_1,...,Q_n\}
\\
&=&\max(z_1M_n^{(1)}, z_2M_n^{(2)},..., z_{N_n}M_n^{(N_n)}, M_n^{(Q)}).
\end{eqnarray*}
Here, 
$M_n^{(Q)}=\max\{Q_1,...,Q_n\}$ relates to the second term in the rhs of (\ref{9}) corresponding to the user preference term $Q_i$ in (\ref{6}). 
Following the same arguments as after (\ref{12}) in Section \ref{Sec2} the statement follows. Really, denoting $k^*=\min\{k,\beta\}$ and $u_n=xn^{1/k^*}$, $x>0$, we get
\begin{eqnarray*}
&&P\{M_{N_n}^*(z)> u_n\}
\\
&=&P\{M_{N_n}^*(z)> u_n, N_n>u_n\}+P\{M_{N_n}^*(z)> u_n, N_n\le u_n\}
\\
&\le & P\{M_{\lceil u_n\rceil}^*(z)> u_n\}+ P\{N_n> u_n\}.
\end{eqnarray*}
On the  other hand, 
\begin{eqnarray*}
&&P\{M_{N_n}^*(z)> u_n\}\ge P\{M_{\lceil u_n\rceil}^*(z)>u_n, N_n>u_n\}
\\
&=& P\{N_n> u_n\}+P\{M_{\lceil u_n\rceil}^*(z)>u_n\}+P\{M_{\lceil u_n\rceil}^*(z)\le u_n, N_n\le u_n\}-1.\nonumber
\end{eqnarray*}
Note that $P\{M_{\lceil u_n\rceil}^*(z)\le u_n, N_n\le u_n\}-1$ tends to zero as $n\to\infty$. Hence, it holds
\begin{eqnarray}\label{24}
P\{M_{N_n}^*(z)> u_n\}&\sim &P\{M_{\lceil u_n\rceil}^*(z)>u_n\}+P\{N_n> u_n\},~~n\to\infty.
\end{eqnarray}
\\
If $k<\beta$ holds, then $M_n^{(Q)}\cdot n^{-1/k^*}\to^P 0$ as $n\to\infty$ since $P\{z_iM_n^{(i)}n^{-1/k}\le x\}\to \exp(-c^{(i)}\theta_iz_i^{k}x^{-k})$, $i=1,2,...$. 
Since $P\{N_n> u_n\}\sim u_n^{-\alpha}\to 0$ as $n\to\infty$ holds, then by 
(\ref{24}) it follows
\begin{eqnarray}\label{31}
\lim_{n\to\infty}P\{M_{N_n}^*(z)n^{-1/k^*}\le x\}&=&\exp\{-c(z)\theta^*(z) x^{-k}\},
\end{eqnarray}
where $\theta^*(z)=\sum_{i=1}^{\infty}c^{(i)}\theta_iz_i^{k}/c(z)$ and $c(z)=\sum_{i=1}^{\infty}c^{(i)}z_i^{k}$.
\\
If $k\ge\beta$ holds, then $z_iM_n^{(i)}\cdot n^{-1/k^*}\to^P 0$, $i=1,2,...$ 
follows since $P\{M_n^{(Q)}n^{-1/\beta}\le x\}\to \exp(-c_q(z^*)^{\beta}x^{-\beta})$ as $n\to\infty$. Thus, we obtain
\begin{eqnarray}\label{32}
\lim_{n\to\infty}P\{M_{N_n}^*(z)n^{-1/k^*}\le x\}&=& \exp(-c_q(z^*)^{\beta}x^{-\beta}).
\end{eqnarray}
Since $\{q_i\}$ are i.i.d., its extremal index is equal to one.
Then by (\ref{31}) and (\ref{32}) the extremal index of $Y_{N_n}^*(z)$ satisfies (\ref{33}) 
irrespectively of $\alpha$.
\\
It remains to show that $Y_{N_n}^*(z)$ and $Y_{N_n}(z)$ have the same extremal index. Similarly to \cite{Gold:13}, we have to derive that
\begin{eqnarray}\label{30}
\lim_{n\to\infty}P\{M_{N_n}(z)n^{-1/k^*}\le x\}&=&\lim_{n\to\infty}P\{M_{N_n}^*(z)n^{-1/k^*}\le x\}.
\end{eqnarray}
Since from the event $\{M_{N_n}(z)\le u_n\}$ it follows $\{M_{N_n}^*(z)\le u_n\}$, and $P\{M_{N_n}(z)\le u_n\}\le P\{M_{N_n}^*(z)\le u_n\}$ holds, we obtain 
similarly to (\ref{26})
\begin{eqnarray}\label{28}0&\le & P\{M_{N_n}^*(z)\le u_n\}-P\{M_{N_n}(z)\le u_n\}\nonumber
\\
&=&  P\{M_{N_n}^*(z)\le u_n\}-P\{M_{N_n}^*(z)\le u_n, M_{N_n}(z)\le u_n\}\nonumber
\\
&=& P\{M_{N_n}^*(z)\le u_n, M_{N_n}(z)> u_n\}\nonumber
\\
&=& P\{M_{N_n}^*(z)\le u_n, M_{N_n}(z)> u_n, N_n> u_n\}\nonumber
\\
&+& P\{M_{N_n}^*(z)\le u_n, M_{N_n}(z)> u_n, N_n\le u_n\}\nonumber
\\
&\le &  P\{N_n> u_n\}+ 
P\{M_{N_n}^*(z)\le u_n, M_{\lfloor u_n\rfloor}(z)> u_n, N_n\le u_n\}\nonumber
\\
&\le &  P\{N_n> u_n\}+\sum_{k=1}^{\lfloor u_n\rfloor}P\{Y_k^*(z)\le u_n, Y_k(z)> u_n\}\nonumber
\\
&=& P\{N_n> u_n\}+\lfloor u_n\rfloor P\{Y_k^*(z)\le u_n, Y_k(z)> u_n\}
\end{eqnarray}
due to the stationarity of $\{Y_k^*(z)\}$ and $\{Y_k(z)\}$.
\\
Completing vectors $(z_1Y_k^{(1)},...,z_{N_k}Y_k^{(N_k)})$ by zeroes up to the maximal dimension $\lfloor u_n\rfloor$, we get
\begin{eqnarray*}P\{Y_k^*(z)\le u_n, Y_k(z)> u_n\}&=& P\{\max(z_1Y_k^{(1)},...,z_{\lfloor u_n\rfloor}Y_k^{(\lfloor u_n\rfloor)}, Q_k)\le u_n,
\\
&& z_1Y_k^{(1)}+...+z_{\lfloor u_n\rfloor}Y_k^{(\lfloor u_n\rfloor)}+Q_k> u_n\}
\end{eqnarray*}
Then (\ref{30}) follows from (\ref{27}) and (\ref{29})
since in (\ref{28})
\begin{eqnarray*}&&
P\{Y_k^*(z)\le u_n, Y_k(z)> u_n\}
=
P\{Y_k(z)> u_n\} P\{Y_k^*(z)\le u_n|Y_k(z)> u_n\}
\end{eqnarray*}
holds.
\end{proof}
\section{Application to Indices of Complex Networks}\label{Sec4}
Theorem \ref{T4} can be applied to PageRank and the Max-linear processes. These processes then have the same tail index and the same extremal index. Theorem \ref{T4} is in the agreement with statements in \cite{Jel:10} and \cite{Vol:10}, namely, that the stationary distribution of PageRank $R=^d \sum_{j=1}^{N_i}A_{j}R_i^{(j)}+Q_i$ is regularly varying and its tail index is determined by a most heavy-tailed distributed term in the triple $(N_i, Q_i, A_iR_i^{(j)})$. This is derived if all terms in the triple are mutually independent. In contrast, Theorem \ref{T4} is valid for an arbitrary dependence structure between $N_n$ and $\{Y_n^{(j)}\}$ as well as $N_n$ and $Q_n$, and $\{N_i\}$ are not necessarily independent.
The novelty of Theorem \ref{T4}  is that the extremal index of both PageRank and the Max-linear processes is the same and it depends on the tail indices in the couple $(Q_i, A_iR_i^{(j)})$, irrespective of the tail index of $N_i$.
\\
The assumptions of both Theorem \ref{T4} and the statements in \cite{Jel:10} and \cite{Vol:10} do not reflect properly the complicated dependence between node ranks due to the entanglement of links  in a real network. For better understanding let us consider 
the matrix
\[\left(
    \begin{array}{cccccc}
      z_1Y_1^{(1)} & z_2Y_1^{(2)} & ...z_{N_1}Y_1^{(N_1)} & 0 & 0 & Q_1 \\
      z_1Y_2^{(1)} & z_2Y_2^{(2)} & ...z_{N_1}Y_2^{(N_1)} & ... z_{N_2}Y_2^{(N_2)} & 0 & Q_2 \\
      ...& ... & ... & ... & ... & ... \\
      z_1Y_n^{(1)} & z_2Y_n^{(2)} & ...z_{N_1}Y_n^{(N_1)} & ... z_{N_2}Y_n^{(N_2)} & ... z_{N_n}Y_n^{(N_n)} & Q_n \\
    \end{array}
  \right)
\]
\[
\left(
    \begin{array}{cccccc}
      (k, \theta_1) & (k, \theta_2) & ...(k, \theta_{N_1}) & ...(k, \theta_{N_2}) & ... ...(k, \theta_{N_n}) & ...(\beta, (z^*)^{\beta}) \\
    \end{array}
  \right)
\]
corresponding to (\ref{9}) and completed by zeros up to the maximal dimension, let's say $N_n$. Strings of the matrix correspond to  generations of descendants of nodes with numbers $1,2,...,n$. Each column may contain descendants of different nodes having the same extremal index $\theta_i$, $i=1,2,...,N_n$. All columns apart of the last one are identically regularly varying distributed with the same tail index $k$. The columns are mutually independent.
\\
In terms of some network, the conditions of Theorem \ref{T4} imply that ranks of all nodes with incoming links to a root node (i.e. its followers) are mutually independent, but
followers of different nodes may be dependent and, thus,  they are combined into clusters.
The reciprocal of the extremal index approximates the mean cluster size, \cite{Lead:83}.
\\
The statement (\ref{33}) implies that the extremal index of PageRank is equal to $\theta(z)=(1-c)^{\beta}$ if the user preference dominates (i.e. its distribution tail is heavier than the tail of ranks of followers). If the damping factor $c$ is close to one, then $\theta(z)$ is close to zero. The latter means  the huge-sized cluster of nodes around a root-node in the presence of rare teleportations. If $c$ is close to zero, then $\theta(z)$ is close to one due to the independence of frequent teleportations. If $k<\beta$ holds, then 
roughly, the mean size of the cluster is determined by the consolidation of all clusters related to the followers of the underlying root.
\\
In practice, the followers of a node may be linked and their ranks can therefore be dependent. The future work will focus on the extremal index of PageRank process when the terms $\{Y_i^{(j)}\}$ in (\ref{9}) are mutually dependent.

\end{document}